\documentclass[a4paper,pdftex,reqno,10pt]{amsart}%

\usepackage{amssymb,amsmath}
\usepackage{mathptmx}       
\usepackage{helvet}         
\usepackage{courier}        
\usepackage{type1cm}        
\usepackage{url}
\usepackage{algorithm2e}
\usepackage{graphicx}        
\usepackage{multicol}        
\usepackage[bottom]{footmisc}

\usepackage{bm} 
\usepackage[shortlabels]{enumitem} 
\usepackage{ytableau} 
\usepackage{csquotes}

\newcommand{\F}{\mathbb{F}}
\newcommand{\gauss}[3]{\genfrac{[}{]}{0pt}{}{#1}{#2}_{#3}}
\newcommand{\qnumb}[2]{[#1]_{#2}}
\newcommand{\SPG}[2]{\operatorname{PG}(#1,#2)}
\newcommand{\PG}{\operatorname{PG}}

\newcommand{\aspace}{\SPG{v-1}{\mathbb{F}_q}}
\newcommand{\vek}[1]{\mathbf{#1}}

\newcommand{\hdist}{\mathrm{d}}
\newcommand{\hweight}{\mathrm{w}}
\newcommand{\sdist}{\mathrm{d}_{\mathrm{S}}}

\newcommand{\smax}{\mathrm{A}}

\newcommand{\card}[1]{\##1}
\newcommand{\snumb}[3]{s_{#3}(#1,#2)} 
\newcommand{\frobenius}[2]{\mathrm{F}_{#2}(#1)} 
\newcommand{\frobeniusproj}[2]{\bar{\mathrm{F}}_{#2}(#1)} 

\def\llceil{\lceil\kern-3.2pt\lceil}
\def\rrceil{\rceil\kern-3.2pt\rceil}
\def\llfloor{\lfloor\kern-3.2pt\lfloor}
\def\rrfloor{\rfloor\kern-3.2pt\rfloor}
\def\leftllceil{\left\lceil\kern-3.2pt\left\lceil}
\def\rightrrceil{\right\rceil\kern-3.2pt\right\rceil}
\def\leftllfloor{\left\lfloor\kern-3.2pt\left\lfloor}
\def\rightrrfloor{\right\rfloor\kern-3.2pt\right\rfloor}
\def\bigllfloor{\bigg\lfloor\kern-3.2pt\bigg\lfloor}
\def\bigrrfloor{\bigg\rfloor\kern-3.2pt\bigg\rfloor}

\newtheorem{theorem}{Theorem}
\newtheorem{proposition}{Proposition}
\newtheorem{lemma}{Lemma}
\newtheorem{corollary}{Corollary}
\theoremstyle{definition}
\newtheorem{definition}{Definition}
\newtheorem{remark}{Remark}
\newtheorem{example}{Example}

\makeindex             
\begin{document}

\title{On the lengths of divisible codes}

\author{Michael Kiermaier and Sascha Kurz}
\address{Michael Kiermaier,University of Bayreuth, 95440 Bayreuth, Germany}
\email{michael.kiermaier@uni-bayreuth.de}
\address{Sascha Kurz, University of Bayreuth, 95440 Bayreuth, Germany}
\email{sascha.kurz@uni-bayreuth.de}

\abstract{In this article, the effective lengths of all $q^r$-divisible linear codes over $\F_q$ with a non-negative integer $r$ are determined.
For that purpose, the $S_q(r)$-adic expansion of an integer $n$ is introduced.
It is shown that there exists a $q^r$-divisible $\F_q$-linear code of effective length $n$ if and only if the leading coefficient of the $S_q(r)$-adic expansion of $n$ is non-negative.
Furthermore, the maximum weight of a $q^r$-divisible code of effective length $n$ is at most $\sigma q^r$, where $\sigma$ denotes the cross-sum of the $S_q(r)$-adic expansion of $n$.

This result has applications in Galois geometries.
A recent theorem of N{\u{a}}stase and Sissokho on the maximum size of a partial spread follows as a corollary.
Furthermore, we get an improvement of the Johnson bound for constant dimension subspace codes.
}}
\maketitle

\section{Introduction}
A linear code $C$ is said to be \emph{$\Delta$-divisible} with $\Delta\in\mathbb{Z}_{\geq 1}$ if all its weights are multiples of $\Delta$.
Divisible codes have been introduced by Ward in~1981 \cite{ward1981divisible}, see \cite{ward2001divisible_survey} for a survey. 
There are relations to self-orthogonal codes \cite{ward1981divisible,ward2001divisible_survey}, Griesmer-optimal codes \cite{Ward-1998-JCTSA83[1]:79-93,ward2001divisible_survey} and, as it will be pointed out in this article, to certain configurations in Galois geometries.
The main case of interest is that $\Delta$ is a power of the characteristic of the base field.%
\footnote{By \cite[Th.~1]{ward1981divisible}, for $\Delta = p^e d$ with $p$ the characteristic of the base field $\F_q$ and $p \nmid d$, each full-length $\Delta$-divisible $\F_q$-linear code is the $d$-fold repetition of a $p^e$-divisible $\F_q$-linear code.}

The \enquote{divisible code bound} of~\cite{ward1992bound,ward2001divisible} gives an upper bound on the dimension of a divisible code.
In this article, we focus on the lengths of $q^r$-divisible $\F_q$-linear codes, without any restriction on the dimension.
As the length of a divisible can always be increased by adding an arbitrary number of all-zero coordinates, it is natural to look at the \emph{effective length}, which is the number of coordinates which are not all-zero.
Codes without all-zero coordinate are called \emph{full-length}.

For a fixed prime power $q$, a non-negative integer $r$ and $i\in\{0,\ldots,r\}$, we define
    \[
	    \snumb{r}{i}{q}
	    := q^i\cdot\qnumb{r-i+1}{q}
	    = \frac{q^{r+1}-q^i}{q-1}
	    =\sum_{j=i}^r q^{j}
	    =q^i + q^{i+1} + \ldots + q^r\text{.}
    \]
The number $\snumb{r}{i}{q}$ is divisible by $q^i$, but not by $q^{i+1}$.
This property allows us to create kind of a positional system upon the sequence of base numbers
\[
	S_q(r) := (\snumb{r}{0}{q}, \snumb{r}{1}{q},\ldots, \snumb{r}{r}{q})\text{.}
\]
As discussed in~Section~\ref{sect:sqradic}, each integer $n$ has a unique \emph{$S_q(r)$-adic expansion}
\begin{equation}
	\label{eq:sqadic}
	n = \sum_{i=0}^r a_i \snumb{r}{i}{q}
\end{equation}
with $a_0,\ldots,a_{r-1}\in\{0,\ldots,q-1\}$ and \emph{leading coefficient} $a_r\in\mathbb{Z}$.
The sum $a_0 + a_1 + \ldots + a_r$ will be called the \emph{cross sum} of the $S_q(r)$-adic expansion of $n$.

Based on the $S_q(r)$-adic expansion we can state our main theorem.

\begin{theorem}
  \label{thm:characterization}
  Let $n\in\mathbb{Z}$ and $r\in\mathbb{N}_0$.
  The following are equivalent:
  \begin{enumerate}[(i)]
  \item\label{thm:characterization:card} There exists a full-length $q^r$-divisible linear code of length $n$ over $\F_q$.
  \item\label{thm:characterization:n_strong} The leading coefficient of the $S_q(r)$-adic expansion of $n$ is non-negative.
  \end{enumerate}
\end{theorem}

The proof of the theorem will use the correspondence between full-length $\F_q$-linear codes and multisets of points in a finite projective geometry over $\F_q$.
As a byproduct of the proof, we get the following theorem on the maximum weight of a divisible code:

\begin{theorem}
  \label{thm:max_weight}
  Let $C$ be a $q^r$-divisible code of effective length $n$.
  Then the maximum weight of $C$ is at most $\sigma q^r$, where $\sigma$ denotes the cross-sum of the $S_q(r)$-adic expansion of $n$.
\end{theorem}

This article is structured as follows:
In Section~\ref{sect:prelim}, the necessary preliminaries are provided.
As the geometric counterpart of divisible linear codes, divisible multisets of points are discussed in Section~\ref{sect:divisible}.
The $S_q(r)$-adic expansion of an integer is introduced in Section~\ref{sect:sqradic}.
The proof of the two stated theorems follows in Section~\ref{sect:main}, which also contains the determination (Proposition~\ref{prop:frobenius}) of the largest integer which is not realizable as the effective length of a $q^r$-divisible $\F_q$-linear code.
In analogy to the \emph{Frobenius Coin Problem}, these numbers will be denoted by $\frobenius{r}{q}$.
In Section~\ref{sect:rounding}, a notion of sharpened rounding will be studied, which is based on the existence of certain divisible codes.
It is a preparation for Section~\ref{sect:application}, where two applications of Theorem~\ref{thm:characterization} in Galois geometry will be presented.
In Section~\ref{subsect:nastase_sissokho}, it is demonstrated that a recent result of N{\u{a}}stase and Sissokho on the maximum size of a partial spread follows as a corollary from Theorem~\ref{thm:characterization}.
In Section~\ref{subsect:johnson}, we get an improvement of the Johnson bound for constant dimension subspace codes. In many cases, this leads to the sharpest known upper bound on the size of a constant dimension subspace code.
Section~\ref{sect:linear_programming} analyses the relation of Theorem~\ref{thm:characterization} to the linear programming bound, which is based on the MacWilliams equations.
In Section~\ref{sect:conclusion}, we conclude with the discussion of two related open problems.

\section{Preliminaries}
\label{sect:prelim}
In this article, $q$ denotes a prime power $>1$ and $V$ an $\F_q$-vector space of finite dimension $v$.
Ordered by inclusion, the set of all $\F_q$-subspaces of $V$ forms a finite modular geometric lattice with meet $X\wedge Y=X\cap Y$, join $X\vee Y=X+Y$, and rank function $X\mapsto\dim(X)$.
This \emph{subspace lattice} of $V$ is also known as the \emph{projective geometry} $\PG(V)$.
Up to isomorphism, $\PG(V)$ only depends on the order $q$ of the base field and the (\emph{algebraic}) dimension $v$, justifying the notion \emph{projective geometry} $\PG(v-1,q)$ of (\emph{geometric}) dimension $v-1$ over $\F_q$.
A $k$-dimensional subspace of the $\F_q$-vector space $V$ will simply be called \emph{$k$-subspace}.
The set of all $k$-subspaces of $V$ will be denoted by $\gauss{V}{k}{q}$.
Its cardinality is given by the Gaussian binomial coefficient
\[
\gauss{v}{k}{q} =
\#\gauss{V}{k}{q} =
\begin{cases}
	\frac{(q^v-1)(q^{v-1}-1)\cdots(q^{v-k+1}-1)}{(q^k-1)(q^{k-1}-1)\cdots(q-1)} & \text{if }0\leq k\leq v\text{;}\\
	0 & \text{otherwise.}
\end{cases}
\]
Furthermore, we use the abbreviation $\qnumb{v}{q} = \gauss{v}{1}{q}$, which is the $q$-analog of the number $v$.
As usual, $1$-subspaces are called \emph{points} and $(v-1)$-subspaces are called \emph{hyperplanes} of $\PG(V)$.

The theory of the finite projective geometries $\PG(v-1,q)$ is known as \emph{Galois geometry}.
As the subspace lattice of a $v$-dimensional $\F_q$-vector space is commonly seen as the $q$-analog of the subset lattice of a finite $v$-element set, Galois geometry can also be seen as \emph{$q$-analog combinatorics}.

A multiset $\mathcal{S}$ on a base set $X$ can be identified with its characteristic function $\chi_X : X \to \mathbb{N}_0$, mapping $x$ to the multiplicity of $x$ in $\mathcal{S}$.
The \emph{cardinality} of $\mathcal{S}$ is $\#\mathcal{S} = \sum_{x\in X} \chi_{\mathcal{S}}(x)$.
$\mathcal{S}$ may also be called a \emph{$(\#\mathcal{S})$-multiset}.
The multiset union $\mathcal{S} \uplus \mathcal{S}'$ of two multisets $\mathcal{S}$ and $\mathcal{S}'$ is given by the sum $\chi_{\mathcal{S}} + \chi_{\mathcal{S}'}$ of the corresponding characteristic functions.
The $q$-fold repetition $q\mathcal{S}$ of a multiset $\mathcal{S}$ is given by the characteristic function $q\chi_{\mathcal{S}}$.

A multiset $\mathcal{S}$ is called \emph{spanning} in $V$ if $\langle \mathcal{S}\rangle_{\F_q} = V$.
For a multiset of points $\mathcal{P}$ in $\PG(V)$ and a hyperplane $H \leq V$, we define the restricted multiset $\mathcal{P} \cap H$ via its characteristic function
\[
	\chi_{\mathcal{P} \cap H}(P) = 
	\begin{cases}
	\chi_{\mathcal{P}}(P) & \text{if }P \leq H\text{;} \\
	0 & \text{otherwise.}
	\end{cases}
\] 
Then $\#(\mathcal{P} \cap H) = \sum_{P\in\gauss{H}{1}{q}} \chi_{\mathcal{P}}(P)$.

It is well-known (see, e.g., \cite[Prop.~1]{tsfasman-vladut95} and \cite{dodunekov1998codes}) that
the relation $C\to\mathcal{C}$, associating with a full-length linear
$[n,v]$ code $C$ over $\F_q$ the $n$-multiset $\mathcal{C}$ of points
in $\aspace$ defined by the columns of any generator matrix, induces a
one-to-one correspondence between classes of \mbox{(semi-)linearly}
equivalent spanning multisets and classes of \mbox{(semi-)linearly}
equivalent full-length linear codes.  The importance
of the correspondence lies in the fact that it relates
coding-theoretic properties of $C$ to geometric or combinatorial
properties of $\mathcal{C}$ via
\begin{equation}
  \label{eq:hweight-npoints}
  \hweight(\vek{a}G)=n-\card{\{1\leq j\leq
  n;\vek{a}\cdot\vek{g}_j=0\}}
  =n-\card{(\mathcal{C}\cap\vek{a}^\perp)},
\end{equation}
where $\hweight$ denotes the Hamming weight,
$G=(\vek{g}_1\mid\dots\mid\vek{g}_n)\in\F_q^{v\times n}$ a generating
matrix of $C$, $\vek{a}\cdot\vek{b}=a_1b_1+\dots+a_vb_v$, and
$\vek{a}^\perp$ is the hyperplane in $\aspace$ with equation
$a_1x_1+\dots+a_vx_v=0$.
In the usual coding theoretic setting, the Hamming weight depends on the chosen basis, as the standard basis vectors are exactly the vectors of Hamming weight $1$.
In contrast to that, the geometric setting provides a basis-free approach to linear codes.

\section{Divisible multisets of points}
\label{sect:divisible}

The geometric counterpart of full-length divisible linear codes are divisible multisets of points:

\begin{definition}
  Let $\mathcal{P}$ be a multiset of points in $V$ and $r\in\{0,\ldots,v-1\}$.
  If
  \[
	  \#(\mathcal{P}\cap H) \equiv \#\mathcal{P} \pmod {q^{r}}
  \]
  for every hyperplane $H \leq V$, then $\mathcal{P}$ is called \emph{$q^r$-divisible}.
\end{definition}

If we speak of a $q^r$-divisible multiset $\mathcal{P}$ of points without specifying the ambient space $V$ or its dimension $v$, we assume that the points in $\mathcal{P}$ are contained in an ambient space $V$ of a suitable finite dimension $v$.
This is justified by the following lemma:

\begin{lemma}
	\label{lem:ambient_space_unwichtig}
	Let $V_1 < V_2$ be $\F_q$-vector spaces and $\mathcal{P}$ a multiset of points in $V_1$.
	Then $\mathcal{P}$ is $q^r$-divisible in $V_1$ if and only if $\mathcal{P}$ is $q^r$-divisible in $V_2$.
\end{lemma}

\begin{proof}
	Assume that $\mathcal{P}$ is $q^r$-divisible in $V_1$.
	Let $H$ be a hyperplane of $V_2$.
	Then $\#(\mathcal{P} \cap H) = \#(\mathcal{P} \cap (H \cap V_1))$.
	$H \cap V_1$ is either $V_1$ or a hyperplane in $V_1$.
	In the first case, the expression equals $\#\mathcal{P}$, and in the second case, it is congruent to $\#\mathcal{P}\pmod{q^r}$ by $q^r$-divisibility of $\mathcal{P}$ in $V_1$.

	Now assume that $\mathcal{P}$ is $q^r$-divisible in $V_2$, and let $H'$ be a hyperplane of $V_1$.
	There is a hyperplane $H$ in $V_2$ such that $H \cap V_1 = H'$.
	So $\#(\mathcal{P}\cap H') = \#(\mathcal{P}\cap H) \equiv \#\mathcal{P}\pmod{q^r}$ by $q^r$-divisibility of $\mathcal{P}$ in $V_2$.
\end{proof}

\begin{lemma}
	\label{lem:qr-div-basic}
	\begin{enumerate}[(a)]
		\item\label{lem:qr-div-basic:subspace} Let $U$ be a $q$-vector space of dimension $k \geq 1$. The set $\gauss{U}{1}{q}$ of points contained in $U$ is $q^{k-1}$-divisible.
		\item\label{lem:qr-div-basic:union} For $q^r$-divisible multisets $\mathcal{P}$ and $\mathcal{P}'$ in $V$, the multiset union $\mathcal{P} \uplus \mathcal{P}'$ is $q^r$-divisible.
		\item\label{lem:qr-div-basic:qfold} The $q$-fold repetition of a $q^r$-divisible multiset $\mathcal{P}$ is $q^{r+1}$-divisible.
	\end{enumerate}
\end{lemma}

\begin{proof}
	For part~\ref{lem:qr-div-basic:subspace}, we take the ambient space $V = U$.
	Let $H$ be a hyperplane of $V$.
	Then $U \cap H$ is a $(k-1)$-space and therefore
	\[
		\#(\gauss{U}{1}{q}\cap H) = \qnumb{k-1}{q} \equiv \qnumb{k}{q} = \#\gauss{U}{1}{q}\pmod {q^{k-1}}\text{.}
  \]
  	Parts~\ref{lem:qr-div-basic:union} and~\ref{lem:qr-div-basic:qfold} are clear from looking at the characteristic functions.
\end{proof}

A subspace $U \leq V$ is commonly identified with the set $\gauss{U}{1}{q}$ of points covered by $U$.
With that identification, Lemma~\ref{lem:qr-div-basic}\ref{lem:qr-div-basic:subspace} simply states that every $k$-subspace is $q^{k-1}$-divisible.
The corresponding linear code is the $q$-ary simplex code of dimension $k$.
In the case $\langle \mathcal{P}\rangle_{\F_q} \cap\langle \mathcal{P}'\rangle_{\F_q} = \{\vek{0}\}$, the multiset union in Lemma~\ref{lem:qr-div-basic}\ref{lem:qr-div-basic:union} corresponds to the direct sum of linear codes, and in the case $\langle \mathcal{P}\rangle_{\F_q}  = \langle \mathcal{P}'\rangle_{\F_q}$ it corresponds to the juxtaposition.
The construction in Lemma~\ref{lem:qr-div-basic}\ref{lem:qr-div-basic:qfold} corresponds to the $q$-fold repetition of a linear code.

For $\lambda\in\mathbb{N}_0$ and a multiset $\mathcal{P}$ of points with maximum point multiplicity at most $\lambda$, we define the \emph{$\lambda$-complementary} multiset $\bar{\mathcal{P}}$ by $\chi_{\bar{\mathcal{P}}}(P)=\lambda-\chi_{\mathcal{P}}(P)$ for all $P\in\gauss{V}{1}{q}$.

\begin{lemma}
  \label{lemma_t_complement}
  Let $\lambda\in\mathbb{N}_0$ and $\mathcal{P}$ a multiset of points in $V$ of maximum point multiplicity at most $\lambda$.
  Let $r\in\{0,\ldots,v-1\}$.
  Then $\mathcal{P}$ is $q^r$-divisible if and only its $\lambda$-complement is.
\end{lemma}

\begin{proof}
  By Lemma~\ref{lem:qr-div-basic}\ref{lem:qr-div-basic:subspace}, $\gauss{V}{1}{q}$ is $q^{v-1}$-divisible.
  By $r < v$, it is $q^r$-divisible.
  Now the result follows from $\chi_{\mathcal{P}} + \chi_{\bar{\mathcal{P}}}  = \lambda\chi_{\gauss{V}{1}{q}}$.
\end{proof}




\begin{lemma}
  \label{lemma_heritable}
  Let $\mathcal{P}$ be a $q^r$-divisible multiset of points in $V$ and $U$ a subspace of $V$ of codimension $j\in\{0,\ldots,r\}$.
  Then the restriction $\mathcal{P}\cap U$ is a $q^{r-j}$-divisible multiset in $U$.
\end{lemma}

\begin{proof}
  By induction, it suffices to consider the case $j=1$.
  Let $W$ be a hyperplane of $U$, that is a subspace of $V$ of codimension $2$.
  There are $q+1$ hyperplanes $H_1,\dots,H_{q+1}$ in $V$ containing $W$ ($U$ being one of them).
  From the $q^r$-divisibility of $\mathcal{P}$ we get
  \[
    (q+1)\card{\mathcal{P}}
    \equiv\sum_{i=1}^{q+1}\card{(\mathcal{P}\cap H_i)}
    =q\cdot\card{(\mathcal{P}\cap W)}+\card{\mathcal{P}}
    \pmod{q^r}\text{.}
  \]
  Hence $q\cdot\card{(\mathcal{P} \cap W)} \equiv q\cdot\card{\mathcal{P}} \equiv q\cdot\card{(\mathcal{P} \cap U)}\pmod{q^r}$
  and thus
  \[
  \card{(\mathcal{P} \cap W)}\equiv\card{(\mathcal{P}\cap U)}\pmod{q^{r-1}}\text{.}
  \]
\end{proof}

The restriction of a multiset of points to a hyperplane $H$ corresponds to the residual of a linear code in a codeword associated with $H$.
In the latter form, Lemma~\ref{lemma_heritable} is found in \cite[Lem.~13]{Ward-1998-JCTSA83[1]:79-93}.

We prepare one more lemma for the proof of Theorem~\ref{thm:characterization}, which guarantees the existence of a hyperplane containing not too many points of $\mathcal{P}$ by an averaging argument.

\begin{lemma}
  \label{lemma_average}
  Let $\mathcal{P}$ be a non-empty multiset of points.
  Then there exists a hyperplane $H$ with $\#(\mathcal{P}\cap H)< \frac{\#\mathcal{P}}{q}$. 
\end{lemma}

\begin{proof}
  Let $V$ be a suitable ambient space of $\mathcal{P}$ of finite dimension $v$.
  Summing over all hyperplanes $H$ gives 
  $
    \sum_{H\in\gauss{V}{v-1}{q}} \#(\mathcal{P}\cap H) =\#\mathcal{P} \cdot \qnumb{v-1}{q}
  $,
  so that we obtain on average
  \[
  \frac{\#\mathcal{P} \cdot \qnumb{v-1}{q}}{\qnumb{v}{q}}
  = 
  \frac{\#\mathcal{P} \cdot \qnumb{v-1}{q}}{q \qnumb{v-1}{q} + 1}
  = \#\mathcal{P} \cdot \frac{1}{q+\frac{1}{\qnumb{v-1}{q}}}
  <\frac{\#\mathcal{P}}{q}
  \]
  points of $\mathcal{P}$ per hyperplane.
  Choosing a hyperplane $H$ that minimizes $\#(\mathcal{P}\cap H)$ completes the proof.
\end{proof}

The coding counterpart of Lemma~\ref{lemma_average} is the well-known existence of a codeword of weight $> \frac{q-1}{q} n_{\operatorname{eff}}$, where $n_{\operatorname{eff}}$ denotes the effective length of $C$.

Now we investigate the sizes of $q^r$-divisible multisets.
For fixed $q$ and $r$, an integer $n$ will be called \emph{realizable} if there exists a $q^r$-divisible multiset of points of size $n$.

\begin{lemma}
	\label{lemma:snumb_card}
	For each $r\in\mathbb{N}_0$ and each $i\in\{0,\ldots,r\}$ there is a $q^r$-divisible multiset of points of cardinality 
	$\snumb{r}{i}{q}$.\footnote{The numbers $\snumb{r}{i}{q} = q^i\cdot\qnumb{r-i+1}{q}$ have been defined in the Introduction.}
\end{lemma}

\begin{proof}
	A suitable multiset of points is given by the $q^i$-fold repetition of an $(r-i+1)$-subspace.
\end{proof}

\begin{lemma}
  \label{lemma_sum}
  The set of sizes of $q^r$-divisible multisets of points is closed under addition.
\end{lemma}

\begin{proof}
Assume that the integers $n_1$ and $n_2$ are realizable.
Then there exist $q^r$-divisible multisets $\mathcal{P}_1$ and $\mathcal{P}_2$ of sizes $n_1$ and $n_2$, respectively.
Let $V_1$ and $V_2$ be the respective ambient spaces.
By Lemma~\ref{lem:ambient_space_unwichtig}, the embeddings of $\mathcal{P}_1$ and $\mathcal{P}_2$ in $V_1 \times V_2$ are $q^r$-divisible.
Now by Lemma~\ref{lem:qr-div-basic}\ref{lem:qr-div-basic:subspace}, their multiset union is a $q^r$-divisible multiset of cardinality $n_1 + n_2$.
\end{proof}

As a consequence of the last two lemmas, all $n = \sum_{i=0}^r a_i \snumb{r}{i}{q}$ with $a_i\in\mathbb{N}_0$ are realizable cardinalities of $q^r$-divisible multisets of points.
As $\snumb{r}{r}{q} = q^r$ and $\snumb{r}{0}{q} = 1 + q + q^2 + \ldots + q^r$ are coprime, for fixed $q$ and $r$ there is only a finite set of cardinalities which is not realizable as a $q^r$-divisible multiset.

Our goal is to show Theorem~\ref{thm:characterization}, which says that actually all possible cardinalities are of the above form.

%

\section{The $S_q(r)$-adic expansion}
\label{sect:sqradic}

We are going to show that each integer $n$ has a unique $S_q(r)$-adic expansion as defined in Equation~\eqref{eq:sqadic}, that is
\[
	n = \sum_{i=0}^r a_i \snumb{r}{i}{q}
\]
with $a_0,\ldots,a_{r-1}\in\{0,\ldots,q-1\}$ and $a_r\in\mathbb{Z}$.
The idea is to consider Equation~\eqref{eq:sqadic} modulo $q, q^2,\ldots,q^r$ which gradually determines $a_0, a_1,\ldots,a_{r-1}\in\{0,\ldots,q-1\}$, using that $\snumb{r}{i}{q}$ is divisible by $q^i$, but not by $q^{i+1}$.

For the existence part, we give an algorithm that computes the $S_q(r)$-adic expansion.

\begin{algorithm} \textbf{Algorithm 1} \\
\label{alg:1}
\KwData{$n\in\mathbb{Z}$, field size $q$, exponent $r$}
\KwResult{representation $n=\sum_{i=0}^r a_i \snumb{r}{i}{q}$ with $a_0,\ldots,a_{r-1}\in\{0,\ldots,q-1\}$ and $a_r\in\mathbb{Z}$}
$m\gets n$\\
\For{$i\gets 0$ \KwTo $r-1$}
{
  $a_i\gets m\bmod q$\\
  $m\gets \frac{m-a_i\cdot\qnumb{r-i+1}{q}}{q}$\\
}
$a_r\gets m$
\end{algorithm}

\begin{lemma}
\label{lem:unique_representation}
Let $n\in\mathbb{Z}$ and $r\in\mathbb{N}_0$.
Algorithm~\ref{alg:1} computes the unique $S_q(r)$-adic expansion of $n$.
\end{lemma}

\begin{proof}
First, we check that Algorithm~\ref{alg:1} computes indeed an $S_q(r)$-adic expansion of $n$. 
Note that in the $i$-th loop run ($i\in\{0,\ldots,r-1\})$ after the execution of \enquote{$a_i\gets m\bmod q$} we have $m \equiv a_i \pmod q$, 
so that the updated value of $m$ in the subsequent line is always an integer, and thus $a_r\in\mathbb{Z}$ at the end of the algorithm.
The line \enquote{$a_i\gets m\bmod q$} provides $a_i\in\{0,\ldots,q-1\}$ for all $i\in\{0,\ldots,r-1\}$.
After the $i$-th loop run, we have $n = q^{i+1} m + \sum_{j=0}^i q^j \qnumb{r-j+1}{q}$, which one shows by induction.
Therefore, at the end of the algorithm
\[
    n = q^r a_r + \sum_{j=0}^{r-1} q^j \qnumb{r-j+1}{q} = \sum_{j=0}^r a_j \snumb{r}{j}{q}\text{.}
\]

For uniqueness, assume that there is a different representation $n=\sum_{i=0}^r b_i \snumb{r}{i}{q}$ with $b_0,\ldots,b_{r-1}\in\{0,\ldots,q-1\}$ and $b_r\in\mathbb{Z}$.
Let $t$ be the smallest index $i$ with $a_i\neq b_i$.
Then $\sum_{i=0}^{t-1} a_i \snumb{r}{i}{q} = \sum_{i=0}^{t-1} b_i \snumb{r}{i}{q}$ and thus
\[
	\underbrace{(a_t - b_t)}_{\neq 0} \snumb{r}{t}{q} = \sum_{i=t+1}^{r} (b_i - a_i) \snumb{r}{i}{q}\text{.}
\]
As $\snumb{r}{i}{q}$ is divisible by $q^i$ but not by $q^{i+1}$, the right hand side is divisible by $q^{t+1}$, but the left hand side is not, which is a contradiction.
\end{proof}

\begin{definition}
	Let $n\in\mathbb{Z}$ and $n = \sum_{i=0}^r a_i \snumb{r}{i}{q}$ be its unique $S_q(r)$-adic expansion.
	The number $a_r$ will be called the \emph{leading coefficient} and the number $\sigma = \sum_{i=0}^r a_i$ will be called the \emph{cross sum} of the $S_q(r)$-adic expansion of $n$.
\end{definition}

\begin{example}
	\label{ex:q3r3n137}
	For $q=3$, $r = 3$, we have $S_3(3) = (40, 39, 36, 27)$.
	For $n = 137$, Algorithm~1 computes
	\begin{align*}
		m & \leftarrow 137\text{,}\\
		a_0 & \leftarrow 137\bmod 3 = 2\text{,}\\
		m & \leftarrow \left(137 - 2\cdot \qnumb{4}{3}\right) / 3 = (137 - 2\cdot 40)/3 = 19\text{,}\\
		a_1 & \leftarrow 19\bmod 3 = 1\text{,}\\
		m & \leftarrow \left(19 - 1\cdot \qnumb{3}{3}\right) / 3 = (19 - 1\cdot 13)/3 = 2\text{,}\\
		a_2 & \leftarrow 2 \bmod 3 = 2\text{,}\\
		m & \leftarrow \left(2 - 2\cdot \qnumb{2}{3}\right) / 3 = (2 - 2\cdot 4)/3 = -2\text{,}\\
		a_3 & \leftarrow -2\text{.}
	\end{align*}
	Therefore, the $S_3(3)$-adic expansion of $137$ is
	\[
		137 = 2\cdot 40 + 1\cdot 39 + 2\cdot 36 + (-2)\cdot 27\text{.}
	\]
	The leading coefficient is $a_3 = -2$, and the cross sum is $2 + 1 + 2 + (-2) = 3$.
\end{example}

\section{Proof of the main theorem}
\label{sect:main}

\begin{proof}[{{Proof of Theorem~\ref{thm:characterization}}}]
We are going to show the geometric version of the theorem.
That is, we replace statement~\ref{thm:characterization:card} by the geometric counterpart \enquote{There exists a $q^r$-divisible multiset of points over $\F_q$ of size $n$}.

The implication \enquote{$\text{\ref{thm:characterization:n_strong}} \Rightarrow\text{\ref{thm:characterization:card}}$} follows from Lemma~\ref{lemma:snumb_card} and Lemma~\ref{lemma_sum}.

The main part of the proof is the verification of \enquote{$\text{\ref{thm:characterization:card}} \Rightarrow\text{\ref{thm:characterization:n_strong}}$}.
The statement is clear for $r = 0$ or $n \leq 0$, so we may assume $r \geq 1$ and $n \geq 1$.

Let $\mathcal{P}$ be a $q^r$-divisible multiset of points of size $n = \#\mathcal{P} \geq 1$.
Let $n=\sum_{i=0}^r a_i \snumb{r}{i}{q}$ with  $a_0,\ldots,a_{r-1}\in\{0,1,\dots,q-1\}$ and $a_r\in\mathbb{Z}$ be the $S_q(r)$-adic expansion of $n$ (see Lemma~\ref{lem:unique_representation}) and $\sigma = \sum_{i=0}^r a_i$ its cross sum.

Let $H$ be a hyperplane in $V$ and $m = \#(\mathcal{P} \cap H)$.
By the $q^r$-divisibility of $\mathcal{P}$ we have $n - m = \tau q^r $ with $\tau\in\mathbb{Z}$.
Using $\snumb{r}{i}{q} = \snumb{r-1}{i}{q} + q^r$, we get
\begin{align}
	\notag m & = n - \tau q^r
	= \sum_{i=0}^{r-1} a_i (\snumb{r-1}{i}{q} + q^r) + a_r q^r -\tau q^r  \\
	\label{eq:m_compact} & = \sum_{i=0}^{r-1} a_i \snumb{r-1}{i}{q} + (\sigma - \tau) q^r \\
	\label{eq:m_unique} & = \sum_{i=0}^{r-2} a_i \snumb{r-1}{i}{q} + (a_{r-1} + q(\sigma - \tau)) q^{r-1}\text{.}
\end{align}
%
%
By Lemma~\ref{lemma_heritable}, $\mathcal{P} \cap H$ is a $q^{r-1}$-divisible multiset of size $m$, and line~\eqref{eq:m_unique} is the $S_q(r-1)$-adic expansion of $m$.
Hence by induction over $r$, we get that $a_{r-1} + q(\sigma - \tau) \geq 0$.
So $q(\sigma - \tau) \geq -a_{r-1} > -q$, implying that $\sigma - \tau > -1$ and thus $\sigma \geq \tau$.

By Lemma~\ref{lemma_average}, we may choose $H$ such that $m < \frac{n}{q}$.
Thus, using the expression for $m$ from line~\eqref{eq:m_compact} together with $q\snumb{r-1}{i}{q} = \snumb{r}{i+1}{q}$ and $\snumb{r}{i}{q} - \snumb{r}{i+1}{q} = q^i$, we get
\begin{align*}
	0
	& < n - qm
	= \sum_{i=0}^{r} a_i \snumb{r}{i}{q} - \sum_{i=0}^{r-1} a_i \snumb{r}{i+1}{q} - (\sigma - \tau) q^{r+1} \\
	& = \sum_{i=0}^{r-1} a_i q^i + a_r q^r - (\sigma - \tau) q^{r+1}
	\leq \sum_{i=0}^{r-1} (q-1) q^i+ a_r q^r 
	= (q^r - 1) + a_r q^r 
	< (1 + a_r) q^r\text{.}
\end{align*}
Therefore $1 + a_r > 0$ and finally $a_r \geq 0$.
\end{proof}

\begin{remark}
	By Theorem~\ref{thm:characterization}, the $S_q(r)$-adic expansion of $n$ provides a certificate not only for the existence, but remarkably also for the non-existence of a $q^r$-divisible multiset of size $n$.

	For instance, the $S_3(3)$-adic expansion $137 = 2\cdot 40 + 1\cdot 39 + 2\cdot 36 + (-2)\cdot 27$ with leading coefficient $-2$ from Example~\ref{ex:q3r3n137} implies immediately that there is no $27$-divisible ternary linear code of effective length $137$.
\end{remark}

\begin{remark}
The proof of Theorem~\ref{thm:characterization} uses the $q^r$-divisibility of $\mathcal{P}$ only in two places:
For the hyperplane $H$ containing less than the average number of points, and for invoking Lemma~\ref{lemma_heritable}, telling us that the restriction of $\mathcal{P}$ to this hyperplane $H$ is $q^{r-1}$-divisible.
Restricting the requirements to what was actually needed in the proof, let us call a multiset $\mathcal{P}$ of points \emph{weakly $q^r$-divisible} if $r = 0$ or if there is a hyperplane $H$ such that $\#(\mathcal{P}\cap H)<\frac{\#\mathcal{P}}{q}$ and $\#\mathcal{P}\equiv \#(\mathcal{P}\cap H)\pmod {q^r}$ and $\mathcal{P}\cap H$ is weakly $q^{r-1}$-divisible.
The statement of Theorem~\ref{thm:characterization} is still true for weakly $q^r$-divisible multisets of points.

There are many more weakly $q^r$-divisible multisets of points than $q^r$-divisible ones.
As an example, any multiset $\mathcal{P}$ of points of size $\#\mathcal{P} = q$ in the projective line $\PG(\F_q^2)$ is weakly $q$-divisible:
Since $\qnumb{2}{q} = q+1 > q$, the projective line contains a point $P$ not contained in $\mathcal{P}$ which provides a suitable hyperplane $H$ for the definition.
The only $q$-divisible multiset of this type is a single point of multiplicity $q$.
\end{remark}

\begin{proof}[{{Proof of Theorem~\ref{thm:max_weight}}}]
	The above proof shows that if $\mathcal{P}$ is a non-empty $q^r$-divisible multiset of size $n$ and $\sigma$ is the cross sum of the $S_q(r)$-adic expansion of $n$, we have $\#\mathcal{P} - \#(\mathcal{P} \cap H) = \tau q^r$ with $\tau \leq \sigma$ for every hyperplane $H$.
	In other words, the maximum weight of a full-length $q^r$-divisible linear code of length $n$ over $\F_q$ is at most $\sigma q^r$.
\end{proof}

\begin{example}
The $S_2(3)$-adic expansion of $n = 59$ is $1\cdot 15 + 0\cdot 14 + 1\cdot 12 + 4\cdot 8$, with cross sum $\sigma = 1 + 0 + 1 + 4 = 6$.
Therefore by Theorem~\ref{thm:max_weight}, the codewords of an $8$-divisible code of effective length $59$ are of weight at most $6\cdot 8 = 48$.
This reasoning is the first step in the proof that there is no projective $8$-divisible binary linear code of length $59$ in~\cite{arXiv1812_05957}.
\end{example}

\begin{example}
In algebraic geometry, a \emph{nodal surface} is a surface in the complex projective space whose only singularities are nodes.
An old problem asks for the maximum number $\mu(s)$ of nodes a nodal surface of given degree $s$ can have \cite{Basset-1906-Nature73:246}.
This problem has been solved only for $s \leq 6$.
The answer in the largest settled case is $\mu(6) = 65$.
The lower bound $\mu(6) \geq 65$ is realized by Barth's sextic \cite{Barth-1996-JAlgebraicGeom5[1]:173-186} and the sextics in the $3$-parameter series in \cite[Th.~5.5.9]{Pettersen-1998-Thesis}.

For the upper bound $\mu(6) \leq 65$, coding theoretic arguments have been used.
Each nodal surface comes with its \emph{even sets of nodes}, which are the codewords of a certain binary linear code $C$ assigned to the nodal surface.
The length $n$ of $C$ is the number of nodes, and $C$ is known to be $4$-divisible if $s$ is odd and $8$-divisible if $s$ is even.
In the case $s = 6$, we have $\dim(C) \geq n-53$, and the nonzero weights of the $8$-divisible code $C$ are contained in $\{24,32,40,56\}$ \cite{Catanese-Tonoli-2007-JEurMathSocJEMS9[4]:705-737}.
For $n = 66$, we get $\dim(C) \geq 13$, which has been shown to be impossible~\cite{Jaffe-Ruberman-1997-JAlgebraicGeom6[1]:151-168}.

It is an open problem to classify the codes $C$ which arise from a nodal sextic having the record number $65$ of nodes.
The $S_2(3)$-adic expansion of $65$ is $1\cdot 15 + 1 \cdot 14 + 1\cdot 12 + 3\cdot 8$ with cross sum $\sigma = 1 + 1 + 1 + 3 = 6$.
If $C$ is full-length, by Theorem~\ref{thm:max_weight} the weights in $C$ are at most $6\cdot 8 = 48$.
So in this case, weight $56$ is not possible and hence all nonzero weights of $C$ are contained in $\{24,32,40\}$.
\end{example}

In analogy to the \emph{Frobenius Coin Problem}, cf.~\cite{brauer1942problem},
we define $\frobenius{r}{q}$ as the smallest integer such that
a $q^r$-divisible multiset of cardinality $n$ exists for all integers
$n>\frobenius{r}{q}$. 
In other words, $\frobenius{r}{q}$ is the largest integer which is not realizable as the size of a $q^r$-divisible multiset of points over $\F_q$.
If all non-negative integers are realizable then $\frobenius{r}{q} = -1$, which is the case for $r = 0$.

\begin{proposition}
  \label{prop:frobenius}
  For every prime power $q$ and $r\in\mathbb{N}_0$ we have
  \[
	  \frobenius{r}{q}= r\cdot q^{r+1} - \qnumb{r+1}{q} = rq^{r+1} - q^r - q^{r-1} - \ldots - 1\text{.}
  \]
\end{proposition}

\begin{proof}
	By Theorem~\ref{thm:characterization}, $\frobenius{r}{q}$ is the largest integer $n$ whose $S_q(r)$-adic expansion $n = \sum_{i=0}^{r-1}a_i\snumb{r}{i}{q} + a_r q^r$ has leading coefficient $a_r < 0$.
	Clearly, this $n$ is given by $a_0 = \ldots = a_{r-1} = q-1$ and $a_r = -1$, such that
	\begin{multline*}
		\frobenius{r}{q}
		= \sum_{i=0}^{r-1} (q-1)\snumb{r}{i}{q} - q^r
		= \sum_{i=0}^{r-1} (q^{r+1} - q^i) - q^r \\
		= rq^{r+1} - \frac{q^r - 1}{q-1} - q^r
		= rq^{r+1} - \frac{q^{r+1}-1}{q-1}\text{.}
	\end{multline*}
\end{proof}

\section{Sharpened rounding}
\label{sect:rounding}

As a preparation for the applications in Galois geometries, we introduce the following notions of sharpened rounding, which are based on the existence of certain divisible codes.

\begin{definition}
	\label{def:divisible_gauss_bracket}
  For $a\in\mathbb{Z}$ and $b\in\mathbb{Z}\setminus\{0\}$ let $\llfloor a/b \rrfloor_{q^r}$ be the maximal $n\in\mathbb{Z}$ such that there exists a $q^r$-divisible $\F_q$-linear code of effective length $a-nb$.
  If no such code exists for any $n$, we set $\llfloor a/b \rrfloor_{q^r} = -\infty$.
  Similarly, let $\llceil a/b\rrceil_{q^r}$ denote the minimal $n\in\mathbb{Z}$ such that there exists a $q^r$-divisible $\F_q$-linear code of effective length $nb-a$. 
  If no such code exists for any $n$, we set $\llceil a/b\rrceil_{q^r} = \infty$
\end{definition}

\begin{remark}
	\label{rem:divisible_gauss_bracket}
	\begin{enumerate}[(a)]
		\item\label{rem:divisible_gauss_bracket:formal}
		Note that the symbols $\llfloor a/b \rrfloor_{q^r}$ and $\llceil a/b \rrceil_{q^r}$ encode the four values $a$, $b$, $q$ and $r$. Thus, the fraction $a/b$ is a formal fraction, and the power $q^r$ is a formal power.
		\item We have
		\[
		    \llfloor 0/b\rrfloor_{q^r} = \llceil 0/b\rrceil_{q^r} = 0
		\]
		and
		\begin{multline*}
		\ldots \leq \llfloor a/b\rrfloor_{q^2} \leq \llfloor a/b\rrfloor_{q^1} \leq \llfloor a/b \rrfloor_{q^0} = \lfloor a/b \rfloor \\
		    \leq a/b \leq \lceil a/b\rceil = \llceil a/b \rrceil_{q^0} \leq \llceil a/b\rrceil_{q^1} \leq \llceil a/b\rrceil_{q^2} \leq \ldots
		\end{multline*}
	\end{enumerate}
\end{remark}

\begin{lemma}
	\label{lem:sharpened_rounding:delta}
	Let $a\in \mathbb{Z}$ and $b \in \mathbb{Z}_{\geq 0}$.
	Then $\lfloor a/b\rfloor - \llfloor a/b\rrfloor_{q^r} \leq \lceil\frac{\frobenius{r}{q} + 1}{b}\rceil$ and $\llceil a/b \rrceil_{q^r} - \lceil a/b\rceil \leq \lceil\frac{\frobenius{r}{q} + 1}{b}\rceil$.
\end{lemma}

\begin{proof}
	By Proposition~\ref{prop:frobenius}, there exists a $q^r$-divisible $\F_q$-linear code of effective length $a - nb$ for all $n\in\mathbb{Z}$ with $a - nb \geq F_q(r) + 1$ or equivalently $n \leq \frac{a - (F_q(r) + 1)}{b}$.
	Therefore, $\llfloor a/b\rrfloor_{q^r} \geq \lfloor\frac{a - (F_q(r) + 1)}{b}\rfloor$ and $\lfloor a/b\rfloor - \llfloor a/b\rrfloor_{q^r} \leq \lceil\frac{F_q(r) + 1}{b}\rceil$.
	The second inequality is shown similarly.
\end{proof}

\begin{remark}
	\label{rem:sharpened_rounding:computation}
	For $a\in\mathbb{Z}$ and $b\in\mathbb{Z}_{\geq 1}$, Theorem~\ref{thm:characterization} and Lemma~\ref{lem:sharpened_rounding:delta} suggest the following method for the computation of $\llfloor a/b\rrfloor_{q^k}$:
	For all $n\in\{\lfloor a/b\rfloor-\lceil\frac{\frobenius{r}{q} + 1}{b}\rceil,\ldots,\lfloor a/b\rfloor\}$, use Algorithm~\ref{alg:1} to compute the leading coefficient of the $S_q(r)$-adic expansion of $a - nb$.
	By definition, $\llfloor a/b\rrfloor_{q^k}$ is the largest of these $n$ whose leading coefficient is non-negative.
	Similarly, $\llceil a/b\rrceil_{q^k}$ is the smallest $n\in\{\lceil a/b\rceil,\ldots,\lceil a/b\rceil+\lceil\frac{\frobenius{r}{q} + 1}{b}\rceil\}$ such that the leading coefficient of the $S_q(r)$-adic expansion of $nb - a$ is positive.
\end{remark}

\begin{lemma}
	\label{lem:sharpened_rounding_unique}
	Let $a\in\mathbb{Z}$ and $b\in\mathbb{Z}_{\geq 1}$ such that there exists a $q^r$-divisible $\F_q$-linear code of effective length $b$.
	Then $\llfloor a/b\rrfloor_{q^k}$ is the unique $n\in\mathbb{Z}$ with the property that there exists a $\F_q$-linear code of effective length $a - nb$, but none of effective length $a - (n+1)b$.
	Similarly, $\llceil a/b\rrceil_{q^k}$ is the unique $n\in\mathbb{Z}$ with the property that there exists a $\F_q$-linear code of effective length $nb - a$, but none of effective length $(n-1)b - a$.
\end{lemma}

\begin{proof}
	By Lemma~\ref{lemma_sum}, the existence of a $q^r$-divisible multiset of points of size $a - nb$ implies the existence of $q^r$-divisible multisets of all sizes $a - mb = (a - nb) + (n - m)b$ with integers $m \leq n$.
	This implies the claim for $\llfloor a/b\rrfloor_{q^k}$.
	The complementary statement for $\llceil a/b\rrceil_{q^k}$ is done analogously.
\end{proof}

\begin{remark}
	\label{rem:sharpened_rounding:computation_improved}
	Lemma~\ref{lem:sharpened_rounding_unique} allows a significant speed-up of the computation strategy for $\llfloor a/b\rrfloor_{q^k}$ discussed in Remark~\ref{rem:sharpened_rounding:computation}:
	Now, a binary search algorithm may be used to find the unique $n$ in the interval $\{\lfloor a/b\rfloor-\lceil\frac{\frobenius{r}{q} + 1}{b}\rceil,\ldots,\lfloor a/b\rfloor\}$ such that the $S_q(r)$-adic expansion of $a-nb$ has a non-negative leading coefficient, but $a - (n+1)b$ has a negative one.
	Thus, the number of needed computations of $S_q(r)$-adic expansions gets logarithmized.
	Again, $\llceil a/b\rrceil_{q^k}$ can be treated similarly.
	
	We leave it as an open problem to study further improvements for the computation of $\llfloor a/b\rrfloor_{q^k}$ and $\llceil a/b\rrceil_{q^k}$.
\end{remark}

\section{Application of divisible codes in Galois geometry}
\label{sect:application}

The connection between divisible codes and Galois geometries is based on the following lemmas.

\begin{lemma}
\label{lem:union_subspaces}
  Let $\mathcal{U}$ be a multiset of subspaces of $V$ and $\mathcal{P} = \uplus_{U\in\mathcal{U}} \gauss{U}{1}{q}$ the \emph{associated multiset 
  of points}.\footnote{In the expression $\biguplus_{U\in\mathcal{U}}$, the subspace $U$ is repeated according to its multiplicity in the multiset $\mathcal{U}$.}
  Let $k$ be the smallest dimension among the subspaces in $\mathcal{U}$.
  If $k \geq 1$, then the multiset $\mathcal{P}$ is $q^{k-1}$-divisible.
\end{lemma} 

\begin{proof}
	Apply Lemma~\ref{lem:qr-div-basic}\ref{lem:qr-div-basic:subspace} and~\ref{lem:qr-div-basic:union}.
\end{proof}

We would like to point out the following important special case of Lemma~\ref{lem:union_subspaces}.
\begin{lemma}
  Let $k\in\mathbb{Z}_{\geq 1}$ and $\mathcal{U} \subseteq \gauss{V}{k}{q}$.
  Then the associated multiset $\biguplus_{U\in\mathcal{U}} \gauss{U}{1}{q}$ of points is $q^{k-1}$-divisible.
\end{lemma} 

\begin{lemma}
  \label{lem:pack_cover}
  Let $k \in \mathbb{Z}_{\geq 1}$ and $\mathcal{U}$ be a multiset of subspaces in $V$ of dimension $\geq k$.
  \begin{enumerate}[(i)]
    \item\label{lem:pack_cover:pack} If every point of $\PG(V)$ is covered by at most $\lambda$ elements of $\mathcal{U}$, then 
    \[
	\#\mathcal{U}\le \llfloor\lambda\cdot\qnumb{v}{q}/\qnumb{k}{q}\rrfloor_{q^{k-1}}\text{.}
    \]
    \item\label{lem:pack_cover:cover} If every point of $\PG(V)$ is covered by at least $\lambda$ elements in $\mathcal{U}$, then 
    \[
	\#\mathcal{U}\ge \llceil\lambda\cdot\qnumb{v}{q}/\qnumb{k}{q}\rrceil_{q^{k-1}}\text{.}
    \]
  \end{enumerate}  
\end{lemma}

\begin{proof}
\begin{enumerate}
	By Lemma~\ref{lem:union_subspaces}, the associated multiset $\mathcal{P} = \uplus_{U\in\mathcal{U}} \gauss{U}{1}{q}$ of points is $q^{k-1}$-divisible.
	Part~\ref{lem:pack_cover:pack}:
	Let $\bar{\mathcal{P}}$ be the $\lambda$-complementary multiset as in Lemma~\ref{lemma_t_complement}. Then 
             $\#\bar{\mathcal{P}} = \lambda\cdot\qnumb{v}{q} - \#\mathcal{U}\cdot \qnumb{k}{q}$ and by Lemma~\ref{lem:union_subspaces} and Lemma~\ref{lemma_t_complement}, $\bar{\mathcal{P}}$ is $q^{k-1}$-divisible.

  Part~\ref{lem:pack_cover:cover}:
  Let $\mathcal{P}'$ arise from $\mathcal{P}$ by reducing the multiplicity of every point by $\lambda$, in characteristic functions $\chi_{\mathcal{P}'} = \chi_{\mathcal{P}} - \lambda \chi_{\gauss{V}{1}{q}}$.
  By Lemma~\ref{lem:qr-div-basic}\ref{lem:qr-div-basic:subspace}, $\gauss{V}{1}{q}$ is $q^{v-1}$-divisible, and by $k \leq v$, it is $q^{k-1}$-divisible.
  So $\mathcal{P}'$ is $q^{k-1}$-divisible of size $\#\mathcal{U}\cdot\qnumb{k}{q}-\lambda\cdot\qnumb{v}{q}$.             
\end{enumerate}             
\end{proof}

\begin{remark}
\label{rem:pack_cover:computation}
By Lemma~\ref{lem:qr-div-basic}\ref{lem:qr-div-basic:subspace}, there is a $q^{k-1}$-divisible multiset of points of size $\qnumb{k}{q}$, which is the denominator in the expressions $\llfloor\lambda\cdot\qnumb{v}{q}/\qnumb{k}{q}\rrfloor_{q^{k-1}}$ and $\llceil\lambda\cdot\qnumb{v}{q}/\qnumb{k}{q}\rrceil_{q^{k-1}}$ in Lemma~\ref{lem:pack_cover}.
Thus, the improved computation method of Remark~\ref{rem:sharpened_rounding:computation_improved} can be used for the evaluation.
\end{remark}

\begin{remark}
\label{rem:pack_cover:improved}
The divisible point sets in the proof of Lemma~\ref{lem:pack_cover} have the additional property that they exist in the ambient space $V$ of dimension $v$.
This dimension property does not give an improvement of Lemma~\ref{lem:pack_cover}, as by Theorem~\ref{thm:characterization}, all sizes $n$ of $q^{r-1}$-divisible multisets of points are a sum of numbers $\snumb{k-1}{i}{q}$ ($i\in\{0,\ldots,k-1\}$), and by the construction in the proof of Lemma~\ref{lemma:snumb_card}, there always exists a suitable multiset of points in dimension $k \leq v$.

However, in part~\ref{lem:pack_cover:pack} we have the additional property that the maximum point multiplicity is bounded by $\lambda$.
Thus, Lemma~\ref{lem:pack_cover}\ref{lem:pack_cover:pack} could possibly be sharpened by restricting the existence question in Definition~\ref{def:divisible_gauss_bracket} to codes with dimension $\leq v$ and maximal point multiplicity at most $\lambda$.
However, the resulting bounds might be much harder to evaluate than those stated in Lemma~\ref{lem:pack_cover} (see Remark~\ref{rem:pack_cover:computation}).
\end{remark}

\subsection{Upper bounds on the maximum size of partial spreads}
\label{subsect:nastase_sissokho}
Let $V$ be a $v$-dimensional vector space over $\F_q$ and $k\in\{1,\ldots,v\}$.
A \emph{partial $(k-1)$-spread} $\mathcal{S}$ in $\PG(V)$ is a set of $k$-subspaces with pairwise trivial intersection.
In other words, each point is covered by at most one element of $\mathcal{S}$.
The maximum size of a partial $(k-1)$-spread will be denoted by $A_q(v,2k;k)$.\footnote{Partial spreads are special cases of constant 
dimension subspace codes, and the symbol $A_q(v,2k;k)$ matches the notation in that more general setting.}

From our preliminary considerations in this section, we get:
\begin{lemma}
\label{lem:partial_spread_generic_ub}
$A_q(v,2k;k) \leq \llfloor \qnumb{v}{q} / \qnumb{k}{q} \rrfloor_{q^{k-1}}$.
\end{lemma}

\begin{proof}
Apply Lemma~\ref{lem:pack_cover}\ref{lem:pack_cover:pack} with $\lambda = 1$.
\end{proof}

The points which remain uncovered by a partial $(k-1)$-spread $\mathcal{S}$ are called \emph{holes} of $\mathcal{S}$.
The set of holes is precisely the $1$-complementary point set in the proof of Lemma~\ref{lem:pack_cover}\ref{lem:pack_cover:pack}.

\begin{lemma}[{{\cite[Theorem 8(ii)]{honold2016partial}}}]
Let $\mathcal{S}$ be a partial $(k-1)$-spread.
Its set of holes is $q^{k-1}$-divisible.
\end{lemma}

\begin{proof}
The set of holes is the $1$-complementary point set of $\bigcup_{B \in \mathcal{S}} \gauss{B}{1}{q}$, which is $q^{k-1}$-divisible by Lemma~\ref{lem:union_subspaces} and Corollary~\ref{lemma_t_complement}.
\end{proof}

Using the properties of the set of holes, we get the following improvement of Lemma~\ref{lem:partial_spread_generic_ub} along the lines of Remark~\ref{rem:pack_cover:improved}.

\begin{lemma}
\label{lem:partial_spread_generic_ub_improved}
Let $n$ be the largest integer such that there exists a projective $q^{k-1}$-divisible $\F_q$-linear code of dimension $\leq v$ and length $\qnumb{v}{q} - n\qnumb{k}{q}$.
Then $A_q(v,2k;k) \leq n$.
\end{lemma}

\begin{proof}
	The set of holes of a partial $(k-1)$-spread $\mathcal{S}$ is a $q^{k-1}$-divisible set of points in $\PG(V)$ of size $\qnumb{v}{q} - \#\mathcal{S}\cdot \qnumb{k}{q}$.
\end{proof}

\begin{remark}
	\begin{enumerate}[(a)]
		\item Lemma~\ref{lem:partial_spread_generic_ub_improved} is strictly stronger than Lemma~\ref{lem:partial_spread_generic_ub}: We have $\llfloor\qnumb{11}{2} / \qnumb{4}{2}\rrfloor_{2^3} = \llfloor 2047 / 15\rrfloor_{2^3} = 135$ as there is a $2^3$-divisible binary code of effective 
		length $2047 - 135\cdot 15 = 22$, but none of effective length $2047 - 136\cdot 15 = 7$.\footnote{Use Lemma~\ref{lem:sharpened_rounding_unique} with the $S_2(3)$-adic expansions $22 = 0\cdot 15 + 1\cdot 14 + 0\cdot 12 + \underline{1}\cdot 8$ with leading coefficient $1 \geq 0$ and $7 = 1\cdot 15 + 0\cdot 14 + 0\cdot 12 + (\underline{-1})\cdot 8$ with leading coefficient $-1 < 0$.}
		However, there are no \emph{projective} $2^3$-divisible binary codes of effective lengths $2047 - 135\cdot 15 = 22$, $2047 - 134\cdot 15 = 37$ and $2047 - 133\cdot 15 = 52$, but there is such a code of length $2047 - 132\cdot 15 = 67$, see~\cite{arXiv1812_05957}.
		Thus, Lemma~\ref{lem:partial_spread_generic_ub} yields $A_2(11,2\cdot 4;4) \leq 135$ and Lemma~\ref{lem:partial_spread_generic_ub_improved} yields $A_2(11,2\cdot 4;4) \leq 132$, the latter being the best 
		known upper bound on $A_2(11,2\cdot 4;4)$.\footnote{The best known bounds are $129 \leq A_2(11,2\cdot 4;4) \leq 132$.}
		\item While Lemma~\ref{lem:partial_spread_generic_ub} can be evaluated based on computing $S_q(k-1)$-adic expansions as suggested in Remark~\ref{rem:sharpened_rounding:computation_improved}, no effective way is known to evaluate Lemma~\ref{lem:partial_spread_generic_ub_improved}.
		Still, Lemma~\ref{lem:partial_spread_generic_ub} is enough to settle a wide range of parameters of partial spreads, see Corollary~\ref{cor:nastase_sissokho}.
		\item Unfortunately, we don't know a closed formula for the evaluation of $\llfloor \qnumb{v}{q} / \qnumb{k}{q} \rrfloor_{q^{k-1}}$ in Lemma~\ref{lem:partial_spread_generic_ub}. 
		For the parameters not covered by Corollary~\ref{cor:nastase_sissokho}, Corollary~\ref{cor:partial_spread_ub} will give an explicit (though somewhat weaker than Lemma~\ref{lem:partial_spread_generic_ub}) upper bound.
		The approach will be similar to the one in the proof of Corollary~\ref{cor:nastase_sissokho}.
	\end{enumerate}
\end{remark}

For $k \mid v$, it is possible to cover all the points by the existence of spreads and thus $A_q(v,2k;k) = \frac{q^v - 1}{q^k - 1}$.
The more involved situation is $k \nmid v$ where no spread exists.

We write $v = tk + r$ with $r \in \{1,\ldots,k-1\}$ and $t\in\mathbb{Z}$.
Then $t \geq 1$.
In \cite[Th.~4.2]{beutelspacher1975partial}, a construction of a partial $(k-1)$-spread of size $\sum_{i=1}^{t-1} q^{ki + r} + 1 = \frac{q^v - q^{k+r}}{q^k - 1} + 1$ has been given.
This construction implies that $A_q(v,2k;k) \geq \frac{q^v - q^{k+r}}{q^k - 1} + 1$.
From the same article we know that this construction is optimal whenever $r = 1$~\cite[Th.~4.1]{beutelspacher1975partial}.
Recently, it has been shown that the same is true whenever $k\nmid v$ and $\qnumb{r}{q} < k$ \cite[Theorem 5]{nastase2016maximum}.

Now we show that this result is indeed a direct consequence of the classification of realizable lengths of divisible codes in Theorem~\ref{thm:characterization}.


\begin{corollary}[{{\cite[Theorem 5]{nastase2016maximum}}}]
\label{cor:nastase_sissokho}
Assume that $k \nmid v$ and let $v = tk + r$ with $r \in \{1,\ldots,k-1\}$.
For $\qnumb{r}{q} < k$ we have
\[
	A_q(v,2k;k) \leq \frac{q^v - q^{k+r}}{q^k - 1} + 1\text{.}
\]
\end{corollary}

\begin{proof}
	Assume that $\mathcal{S}$ is a partial $(k-1)$-spread of size $\#\mathcal{S} = \frac{q^v - q^{k+r}}{q^k - 1} + 2$.
	By Lemma~\ref{lem:partial_spread_generic_ub}, there is a $q^{k-1}$-divisible $\F_q$-linear code of effective length $n = \qnumb{v}{q} - \#\mathcal{S}\cdot \qnumb{k}{q} = \qnumb{k+r}{q} - 2\qnumb{k}{q}$.
	We have
	\begin{align}
		\label{eq:nastase_sissokho_Sqadic} & \phantom{ = }\sum_{i=0}^{k-2} (q-1) \snumb{k-1}{i}{q} + \left(q\cdot(\qnumb{r}{q} - k + 1) - 1\right) \snumb{k-1}{k-1}{q}\\
		\notag & = \sum_{i=0}^{k-2} (q^k - q^i) - (k-1)q^k - q^{k-1} + q^k\cdot\qnumb{r}{q} \\
		\notag & = - \left(\frac{q^{k-1} - 1}{q-1} + q^{k-1}\right) + \frac{q^{k+r} - q^k}{q-1}
		= \frac{q^{k+r} - 2q^k + 1}{q-1} = n\text{.}
	\end{align}
	So \eqref{eq:nastase_sissokho_Sqadic} is the $S_q(k-1)$-adic expansion of $n$ and by Theorem~\ref{thm:characterization}, its leading coefficient $q\cdot(\qnumb{r}{q} - k + 1) - 1$ is $\geq 0$.
	Equivalently $k \leq \qnumb{r}{q}$, which is a contradiction.
\end{proof}

\begin{remark}
	Combined with the construction in \cite[Th.~4.2]{beutelspacher1975partial}, Corollary~\ref{cor:nastase_sissokho} shows indeed that $A_q(v,2k;k) = \frac{q^v - q^{k+r}}{q^k - 1} + 1$, which is the full statement of \cite[Theorem 5]{nastase2016maximum}.
\end{remark}

Now we apply the same technique to the cases not covered by Corollary~\ref{cor:nastase_sissokho}.

\begin{corollary}
  \label{cor:partial_spread_ub}
  Let $v = tk + r$ with $r \in \{0,\ldots,k-1\}$ and assume that $\qnumb{r}{q} \geq k$.
  Then
  \[
  A_q(v,2k;k) \le \frac{q^v-q^{k+r}}{q^k-1}+q\left(\qnumb{r}{q} - k + 1\right) + 1\text{.}
  \]
\end{corollary}

\begin{proof}
    Let $z = \qnumb{r}{q} - k + 1 \geq 0$ and assume that $\mathcal{S}$ is a partial $(k-1)$-spread of size $\#\mathcal{S} = \frac{q^{v}-q^{k+r}}{q^k-1}+qz+2$.
    Using $(q-1)\sum_{i=0}^{k-2} q^i \qnumb{k-i}{q} = (k-1)q^k - \qnumb{k-1}{q}$, its set $\mathcal{P}$ of holes is $q^{k-1}$-divisible of size 
    \begin{align*}
    \#\mathcal{P} & = \qnumb{k+r}{q}-(qz + 2)\qnumb{k}{q} \\
	& = q^k \cdot\qnumb{r}{q} - \qnumb{k}{q} -zq^k + z - z\qnumb{k}{q} \\
	& = -z q \qnumb{k-1}{q} + q^k(k-1) - \qnumb{k}{q} \\
	& = -z q \qnumb{k-1}{q} + (q-1)\sum_{i=0}^{k-2} q^i \qnumb{k-i}{q}-q^{k-1} \text{.}
    \end{align*}
    Writing $z=\sum_{i=0}^{k-2}b_iq^i$ with $b_i\in\{0,\ldots,q-1\}$ for $0\le i\le k-3$ and $b_{k-2} \in\mathbb{Z}_{\geq 0}$, we further transform this expression into
    \begin{align*}
    	\#\mathcal{P} & = -\sum_{i=0}^{k-3}\left(q^{i+1}\qnumb{k-i-1}{q} + q^k\qnumb{i}{q}\right) b_i + q^{k-1}\qnumb{k-1}{q}b_{k-2} + (q-1)\sum_{i=0}^{k-2} q^i \qnumb{k-i}{q} - q^{k-1} \\
	& = (q-1)\qnumb{k}{q} + \sum_{i=1}^{k-2}q^i \qnumb{k-i}{q}(q-1-b_{i-1}) + q^{k-1}\left(-\sum_{i=0}^{k-3}q\qnumb{i}{q}b_i - \qnumb{k-1}{q}b_{k-2} - 1\right) \\
	& = \sum_{i=0}^{k-1} a_i \snumb{k-1}{i}{q}\text{,}
    \end{align*}
  which is the $S_q(k-1)$-adic expansion of $\#\mathcal{P}$ with $a_0 = q-1$, $a_i = q-1-b_{i-1} \in\{0,\ldots,q-1\}$ for $i\in\{1,\ldots,k-2\}$ and leading coefficient
  \[
	  a_{k-1} = -\left(\sum_{i=0}^{k-3}q\qnumb{i}{q}b_i + \qnumb{k-1}{q}b_{k-2} + 1\right) < 0\text{.}
  \]
  Contradiction.
\end{proof}


\begin{remark}
	Similar upper bounds as in Corollary~\ref{cor:partial_spread_ub} have been published in~\cite[Th.~2.9]{kurz2017packing} and~\cite[Th.~6]{nastase2016maximumII}.
	In contrast to Corollary~\ref{cor:partial_spread_ub}, the former one uses the projectivity of the code given by the hole set.
	For example, it yields $A_2(17,14;7) \leq 1026$, while Lemma~\ref{lem:partial_spread_generic_ub} (which is stronger but less explicit than Corollary~\ref{cor:partial_spread_ub}) only gives $A_2(17,14;7) \leq 1027$.%
	\footnote{Use Lemma~\ref{lem:sharpened_rounding_unique} with the $S_2(6)$-adic expansions $\qnumb{17}{2} - 1027 \cdot \qnumb{7}{2} = 642 = 0\cdot 127 + 1\cdot 126 + 1\cdot 124 + 1\cdot 120 + 1\cdot 112 + 1\cdot 96 +\underline{1}\cdot 64$ with leading coefficient $1 \geq 0$ and $\qnumb{17}{2} - 1028 \cdot \qnumb{7}{2} = 515 = 1\cdot 127 + 0\cdot 126 + 1\cdot 124 + 1\cdot 120 + 1\cdot 112 + 1\cdot 96 + (\underline{-1})\cdot 64$ with leading coefficient $-1 < 0$.}
\end{remark}

\begin{remark}
	We would like to point out that every single known upper bound on the size of a partial spread can be obtained by Lemma~\ref{lem:partial_spread_generic_ub} or Lemma~\ref{lem:partial_spread_generic_ub_improved}.
\end{remark}

\subsection{An improvement of the Johnson bound for constant dimension subspace codes}
\label{subsect:johnson}
The geometry $\aspace$ serves as input and output alphabet of the so-called
\emph{linear operator channel (LOC)} -- a model for information
transmission in coded packet networks subject to noise
\cite{koetter-kschischang08}. The relevant metrics on the LOC are given by the
\emph{subspace distance}
$ d_S(X,Y):=\dim(X+Y)-\dim(X\cap Y)=2\cdot\dim(X+Y)-\dim(X)-\dim(Y)$,
which can also be seen as the graph-theoretic distance in the Hasse
diagram of $\aspace$, and the \emph{injection distance}
$ d_I(X,Y):=\max\left\{\dim(X),\dim(Y)\right\}-\dim(X\cap Y) $.  A set
$\mathcal{C}$ of subspaces of $\F_q^v$ is called a \emph{subspace 
code}.
For $\#\mathcal{C} \geq 2$, the \emph{minimum (subspace)  distance} of $\mathcal{C}$ is given by
$d = \min\{d_S(X,Y) \mid X,Y\in\mathcal{C}, X \neq Y\}$.
If all elements of $\mathcal{C}$ have the same dimension $k$, we call
$\mathcal{C}$ a \emph{constant-dimension code} and denote its parameters as $[v,d,\#\mathcal{C};k]_q$.
Partial spreads are the same as subspace codes of constant dimension $k$ and minimum subspace distance $d = 2k$.
For a constant-dimension code $\mathcal{C}$ we have $d_S(X,Y) = 2d_I(X,Y)$
for all $X,Y\in\mathcal{C}$, so that we can restrict our attention to the
subspace distance, which has to be even. By $\smax_q(v,d;k)$ we denote the maximum possible cardinality of a 
constant-dimension-$k$ code in $\F_q^v$ with minimum subspace distance at least $d$.
Like in the classical case of codes in the Hamming metric, the determination of the exact 
value or bounds for $\smax_q(v,d;k)$ is a central problem. In this paper we will present some 
improved upper bounds. For a broader background we refer to \cite{etzionsurvey,COSTbook} and for the latest 
numerical bounds to the online tables at \url{http://subspacecodes.uni-bayreuth.de} \cite{TableSubspacecodes}.

For a subspace $U \leq \F_q^v$, the orthogonal subspace with respect to some fixed non-de\-ge\-ne\-rate symmetric bilinear form will be denoted $U^\perp$.
It has dimension $\dim(U^\perp) = v-\dim(U)$.
For $U,W\leq\F_q^v$, we get that $\sdist(U,W)=\sdist(U^\perp,W^\perp)$. 
So, $\smax_q(v,d;k)=\smax_q(v,d;v-k)$ and we can assume $0\le k\le \frac{v}{2}$ in the following. If $d>2k$, 
then $\smax_q(v,d;k)=1$.
Furthermore, we have $\smax_q(v,2;k)=\gauss{v}{k}{q}$.
Things get more interesting for $v,d\ge 4$ and $k\ge 2$.

Let $\mathcal{C}$ be a constant-dimension-$k$ code in $\F_q^v$ with minimum distance $d$. For every point 
$P$, i.e., $1$-subspace, of $\F_q^v$ we can consider the quotient geometry $\PG(\F_q^v/P)$ to deduce that at most 
$\smax_q(v-1,d;k-1)$ elements of $\mathcal{C}$ contain $P$. Since $\PG(\F_q^v)$ contains $\qnumb{v}{q}$ points and 
every $k$-subspace contains $\qnumb{k}{q}$ points, we obtain 
\begin{equation}
  \label{ie_johnson}
  \smax_q(v,d;k)\le \left\lfloor \frac{\qnumb{v}{q}\cdot \smax_q(v-1,d;k-1)}{\qnumb{k}{q}}\right\rfloor
  \text{,}
\end{equation}  
which was named Johnson type bound II 
in \cite{xia2009johnson}. Recursively applied, we obtain
\begin{equation}
  \label{ie_johnson_to_partial_spread}
  \smax_q(v,d;k)\le \left\lfloor \frac{\qnumb{v}{q}}{\qnumb{k}{q}}\cdot\left\lfloor \frac{\qnumb{v-1}{q}}{\qnumb{k-1}{q}}\cdot
  \left\lfloor \dots \cdot\left\lfloor\frac{\qnumb{v'+1}{q}}{\qnumb{d/2+1}{q}}\cdot \smax_q(v',d;d/2) 
  \right\rfloor \dots \right\rfloor \right\rfloor\right\rfloor,
\end{equation}
where $v'=v-k+d/2$.

In the case $d=2k$, any two codewords of $\mathcal{C}$ intersect trivially, meaning that each point of $\PG(\F_q^v)$ is covered by at most a single codeword.
These codes are better known as \emph{partial $k$-spreads}.
If all the points are covered, we have $\#\mathcal{C} = \qnumb{v}{q}/\qnumb{k}{q}$ and $\mathcal{C}$ is called a \emph{$k$-spread}.
From the work of Segre in 1964 \cite[\S VI]{segre1964teoria} we know that $k$-spreads exist if and only 
if $k$ divides $v$. Upper bounds for the size of a partial $k$-spreads are due to Beutelspacher \cite{beutelspacher1975partial} and Drake \& Freeman 
\cite{nets_and_spreads} and date back to 1975 and 1979, respectively. Starting from \cite{kurzspreads} several recent improvements have been obtained. 
Currently the tightest upper bounds, besides $k$-spreads, are given by a list of $21$ sporadic $1$-parametric series and the following 
two theorems stated in \cite{kurz2017packing}:
\begin{theorem}
  \label{main_theorem_1_ps}
  For integers $r\ge 1$, $t\ge 2$, $u\ge 0$, and $0\le z\le \qnumb{r}{q}/2$ with $k=\qnumb{r}{q}+1-z+u>r$ we have
  $\smax_q(v,2k;k)\le lq^k+1+z(q-1)$, where $l=\frac{q^{v-k}-q^r}{q^k-1}$ and $v=kt+r$.   
\end{theorem}

\newcommand{\uu}{\lambda}
\begin{theorem}
  \label{main_theorem_2_ps}
  For integers $r\ge 1$, $t\ge 2$, $y\ge \max\{r,2\}$, 
  $z\ge 0$ with $\uu=q^{y}$, $y\le k$, 
  $k=\qnumb{r}{q}+1-z>r$, $v=kt+r$, and  $l=\frac{q^{v-k}-q^r}{q^k-1}$, we have
  \[
	  \smax_q(v,2k;k)\le 
     lq^k+\left\lceil \uu -\frac{1}{2}-\frac{1}{2}
    \sqrt{1+4\uu\left(\uu-(z+y-1)(q-1)-1\right)} \right\rceil\text{.}
  \]
\end{theorem}
 
The special case $z=0$ in Theorem~\ref{main_theorem_1_ps} covers the breakthrough $\smax_q(kt+r,2k;k)=1+\sum_{s=1}^{t-1}q^{sk+r}$ for 
$0<r<k$ and $k>\qnumb{r}{q}$ by N{\u{a}}stase and Sissokho \cite{nastase2016maximum} from 2016, which itself covers the result of Beutelspacher.  
The special case $y=k$ in Theorem~\ref{main_theorem_2_ps} covers the result by Drake \& Freeman. A contemporary survey of the best known upper bounds 
for partial spreads can be found in \cite{honold2016partial}.  

Using the tightest known upper bounds for the sizes of partial $k$-spreads, there are only two known cases with $d<2k$ where Inequality~(\ref{ie_johnson_to_partial_spread}) is not sharp:
$\smax_2(6,4;3)=77<81$ \cite{hkk77} and $\smax_2(8,6;4) = 257 < 289$ \cite{new_bounds_subspaces_codes,code257}.
For the details how 
the proposed upper bounds for constant-dimension codes relate to Inequality~(\ref{ie_johnson_to_partial_spread}) we refer the 
interested reader to \cite{MR3063504,heinlein2017asymptotic}.     
The two mentioned improvements of Inequality~(\ref{ie_johnson_to_partial_spread}) involve massive computer calculations.
In contrast to that, the improvements in this article are based on a self-contained theoretical argument and do not need any external computations.

\begin{theorem}
  \label{thm:johnson_improved}
  \[
  \smax_q(v,d;k)
	  \le \leftllfloor \frac{\qnumb{v}{q}\cdot \smax_q(v-1,d;k-1)}{\qnumb{k}{q}}\rightrrfloor_{q^{k-1}} \text{.}
  \]
\end{theorem}

\begin{proof}
  Let $\mathcal{C}$ be a $[v,d,\#\mathcal{C};k]_q$ subspace code and $\mathcal{P} = \uplus_{B\in\mathcal{C}} \gauss{B}{1}{q}$ its associated multiset of points.
  As in the reasoning for the Johnson bound~\eqref{ie_johnson}, the maximum point multiplicity of $\mathcal{P}$ is at most $\lambda=\smax_q(v-1,d;k-1)$.
  Lemma~\ref{lem:pack_cover}\ref{lem:pack_cover:pack} concludes the proof.
\end{proof}


\begin{remark}
	Similarly as in Lemma~\ref{lem:partial_spread_generic_ub_improved}, Theorem~\ref{thm:johnson_improved} could possibly be sharpened further in the following way, at the price that the involved numbers are much harder to evaluate:
	Let $n$ be the largest integer such that there exists a $q^{k-1}$-divisible $\F_q$-linear code of dimension $\leq v$, maximum point multiplicity $\leq \smax_q(v-1,d;k-1)$ and length $\smax_q(v-1,d;k-1)\qnumb{v}{q} - n\qnumb{k}{q}$.
Then $\smax_q(v,d;k) \leq n$.
\end{remark}

\begin{remark}
  With $v'=v-k+d/2$, the iterated application of Theorem~\ref{thm:johnson_improved} yields
  \[
  \smax_q(v,d;k)\le \bigllfloor \frac{\qnumb{v}{q}}{\qnumb{k}{q}}\cdot\bigllfloor \frac{\qnumb{v-1}{q}}{\qnumb{k-1}{q}}\cdot
  \bigllfloor \cdots\bigllfloor\frac{\qnumb{v'+1}{q}}{\qnumb{d/2+1}{q}}\cdot \smax_q(v',d;d/2) 
  \bigrrfloor_{q^{d/2-1}} \cdots \bigrrfloor_{q^{k-3}} \bigrrfloor_{q^{k-2}}\bigrrfloor_{q^{k-1}}\text{,}\footnotemark
  \]
  \footnotetext{Expressions of the form $\llfloor\frac{a}{b}\cdot c\rrfloor_{q^r}$ should be read as $\llfloor\frac{a\cdot c}{b}\rrfloor_{q^r}$, compare to Remark~\ref{rem:divisible_gauss_bracket}\ref{rem:divisible_gauss_bracket:formal}.}
  which is an improvement of~\eqref{ie_johnson_to_partial_spread}. 
\end{remark}
  

\begin{example}
So far, the best known upper bound on $\smax_2(9,6;4)$ has been given by the Johnson bound~\eqref{ie_johnson}, using $\smax_2(8,6;3)=34$:
\[
	\smax_2(9,6;4) \leq \left\lfloor\frac{\qnumb{9}{2}}{\qnumb{4}{2}}\cdot\smax_2(8,6;3)\right\rfloor
	= \left\lfloor\frac{2^9-1}{2^4-1}\cdot 34\right\rfloor = 1158\text{.}
\]
To improve that bound by Theorem~\ref{thm:johnson_improved}, we are looking for the largest integer $n$ such that a $q^{k-1}$-divisible multiset of size
\[
	M(n)
	= \qnumb{9}{2}\cdot\smax_2(8,6;3) - n \cdot \qnumb{4}{2}
	= 17374 - 15n
\]
exists.

This question can be investigated with Theorem~\ref{thm:characterization}.
We have $S_2(3) = (15, 14, 12, 8)$.
The $S_2(3)$-adic expansion of $M(1157) = 17374 - 15\cdot 1157 = 19$ is $1\cdot 15 + 0\cdot 14 + 1\cdot 12 + (-1)\cdot 8$.
As the leading coefficient $-1$ is negative, there is no $8$-divisible multiset of points of size $19$ by Theorem~\ref{thm:characterization}.
The $S_2(3)$-adic expansion of $M(1156) = 34$ is $0\cdot 15 + 1\cdot 14 + 1\cdot 12 + 1\cdot 8$.
As the leading coefficient $1$ is non-negative, there exists a $8$-divisible multiset of points of size $34$.
Therefore by Lemma~\ref{lem:sharpened_rounding_unique}
\[
	\smax_2(9,6;4)\le \leftllfloor\frac{\qnumb{9}{2}}{\qnumb{4}{2}}\cdot\smax_2(8,6;3)\rightrrfloor_{2^3}
	= \llfloor 17374 / 15\rrfloor_{2^3}
	=1156\text{,}
\]
which improves the original Johnson bound~\eqref{ie_johnson} by $2$.
\end{example}

\begin{lemma}
The improvement of Theorem~\ref{thm:johnson_improved} over the original Johnson bound~\eqref{ie_johnson} is at most $(q-1)(k-1)$.
\end{lemma}

\begin{proof}
	By Lemma~\ref{lem:sharpened_rounding:delta}, the improvement is at most
	\[
	\left\lceil \frac{F_q(k-1) + 1}{\qnumb{k}{q}}\right\rceil 
	= \left\lceil \frac{(k-1)q^k - \qnumb{k}{q} + 1}{\qnumb{k}{q}}\right\rceil
	= \left\lceil (q-1)(k-1) - 1 + \frac{k}{\qnumb{k}{q}}\right \rceil = (q-1)(k-1)\text{.}
	\]
\end{proof}
 


\begin{proposition}
	\label{prop:v11d6k4}
  For all prime powers $q\ge 2$ we have
  \begin{align*}
	  \smax_q(11,6;4) & \le q^{14}+q^{11}+q^{10}+2q^7+q^6+q^3+q^2-2q+1 \\
	  & = (q^2-q+1)(q^{12}+q^{11}+q^8+q^7+q^5+2q^4+q^3-q^2-q+1)\text{.}
  \end{align*}
\end{proposition} 

\begin{proof}
  Since $10\equiv 1 \pmod 3$ we have $\smax_q(10,6;3)=q^7+q^4+1$.
  Let
  \[
	  M(n) = \qnumb{11}{q}\cdot (q^7 + q^4 + q) - \qnumb{4}{q}\cdot n\text{.}
  \]
  For $n^{\ast} = q^{14}+q^{11}+q^{10}+2q^7+q^6+q^3+q^2-2q+1$ one computes
  \begin{align*}
  	& \phantom{{}={}} M(n^{\ast}+1) \\
	& = 2q^4 - q^3 + q - 1 \\
	& = (q-1)\cdot (q^3+q^2+q+1) +1\cdot (q^3+q^2+q) +(q-1)\cdot (q^3+q^2) + (-2)\cdot q^3
\end{align*}
  where the last expression is the $S_q(3)$-adic expansion.
  As the leading coefficient $-2$ is negative, by Theorem~\ref{thm:characterization} there exists no $q^3$-divisible multiset of size $M(n^{\ast}+1)$.
  Therefore by Theorem~\ref{thm:johnson_improved}
  \[
      \smax_q(11,6;4) \leq\leftllfloor\frac{\qnumb{11}{q}\cdot (q^7 + q^4 + q)}{\qnumb{4}{2}}\rightrrfloor_{q^3}
      \leq n^{\ast}\text{.}
  \]
%
%
%
%
\end{proof}

\begin{remark}
	In the proof of Proposition~\ref{prop:v11d6k4} we have in fact
	\[
	\leftllfloor\frac{\qnumb{11}{q}\cdot (q^7 + q^4 + q)}{\qnumb{4}{2}}\rightrrfloor_{q^3} = n^{\ast}\text{.}
	\]
	To see this, we use Lemma~\ref{lem:sharpened_rounding_unique} and compute
	\begin{align*}
  	& \phantom{{}={}} M(n^{\ast}) \\
	& = 2q^4 + q^2 + 2q \\
	& = 0\cdot (q^3+q^2+q+1) +2\cdot (q^3+q^2+q) +(q-1)\cdot (q^3+q^2) + (q-2)\cdot q^3
	\end{align*}
	where the last expression is the $S_q(3)$-adic expansion.
  As the leading coefficion $q-2$ is non-negative, by Theorem~\ref{thm:characterization} there exists a $q^3$-divisible multiset of size $M(n^{\ast})$.
\end{remark}

\section{Divisible codes and the linear programming method}
\label{sect:linear_programming}

The famous \emph{MacWilliams Identities},
\cite{macwilliams63}
\begin{equation}
  \label{mac_williams_identies}
  \sum_{j=0}^{n-i} {{n-j}\choose i} A_j=q^{k-i}\cdot \sum_{j=0}^i
  {{n-j}\choose{n-i}}A_j^\perp\quad\text{for }0\le i\le n, 
\end{equation} 
relate the weight distributions $(A_i)$, $(A_i^\perp)$ of the (primal)
code $C$ and the dual code
$C^\perp=\{\vek{y}\in\F_q^n;x_1y_1+\dots+x_ny_n=0\text{ for all
}\vek{x}\in C\}$. 
Since the $A_i$ and $A_i^\perp$ count codewords of weight $i$, they have to be non-negative integers. 
In our context we have $A_0=A_0^\perp=1$, $A_1^\perp=0$, and $A_i=0$ for all $i$ that are not divisible by $q^r$. 
Treating the remaining $A_i$ and $A_i^\perp$ as non-negative real variables one can check feasibility via 
linear programming, which is known as the \emph{linear programming method} for the existence of codes, see 
e.g.~\cite{delsarte1972bounds,bierbrauer2005introduction}. 

As demonstrated in e.g.~\cite{honold2016partial}, the average argument of Lemma~\ref{lemma_average} is equivalent to the linear programming 
method applied to the first two MacWilliams Identities, i.e., $i=0,1$. So, the proof of Theorem~\ref{thm:characterization} 
shows that invoking the other equations gives no further restrictions for the possible lengths of divisible codes. 
This is different in the case of partial $k$-spreads, i.e., the determination of $\smax_q(v,2k;k)$. Here the associated multisets of points are indeed sets that correspond to projective linear codes, which are
characterized by the additional condition $\hdist(C^\perp)\geq 3$, i.e., $A_2^\perp=0$. The upper bound of N{\u{a}}stase and Sissokho 
can be concluded from the first two MacWilliams Identities, i.e., the average argument of Lemma~\ref{lemma_average}, see Corollary~\ref{cor:partial_spread_ub}.
Theorem~\ref{main_theorem_1_ps} and Theorem~\ref{main_theorem_2_ps} are based on the 
first three MacWilliams Identities while also the forth MacWilliams Identity is needed 
for the mentioned $21$ sporadic $1$-parametric series listed in \cite{kurz2017packing}.

\section{Conclusion}
\label{sect:conclusion}

We would like to mention the following open questions:
\begin{itemize}
	\item In this article, the lengths of $\Delta$-divisible codes over $\F_q$ have been classified for $\Delta = q^r$ with $r$ a non-negative integer.
	This leaves the cases open where $\Delta$ is only a power of the characteristic of $\F_q$.
	\item As discussed in Remark~\ref{rem:pack_cover:improved}, the applications in Galois geometries might be improved when the restricted point multiplicity of the associated divisible multisets of points is taken into account.
	This condition may further restrict the set of realizable lengths.
	The particular case $\lambda = 1$ corresponds to \emph{projective} linear divisible codes, where a general characterization of the realizable lengths is wide open and appears to be more difficult than in the non-projective case of Theorem~\ref{thm:characterization}.
	For the corresponding Frobenius number the sharpest upper bound in the binary case $q=2$ is $\frobeniusproj{r}{2}\le 2^{2r}-2^{r-1}-1$.
	The lengths of projective $2$-, $4$-divisible and $8$-divisible linear binary codes as well as $3$-divisible linear ternary codes have been completely determined~\cite{projective_divisible_binary_codes,arXiv1812_05957}, but there are open cases already for $(q,\Delta) = (2,16),(3,9),(5,5)$.
\end{itemize}

\section*{Acknowledgement}
The second author was supported in part by the grant KU 2430/3-1 -- Integer Linear Programming Models for 
Subspace Codes and Finite Geometry -- from the German Research Foundation.


\end{document}